\documentclass[12pt]{amsart}
\usepackage{mathrsfs}
\usepackage{amsmath,amssymb,amsbsy,amsfonts,amsthm,latexsym,
                        amsopn,amstext,amsxtra,euscript,amscd,mathrsfs,color}

\usepackage{float}
\usepackage[english]{babel}
\usepackage{url}
\usepackage[colorlinks,linkcolor=blue,anchorcolor=blue,citecolor=blue]{hyperref}

\newtheorem{theorem}{Theorem}
\newtheorem{lemma}[theorem]{Lemma}
\newtheorem{proposition}[theorem]{Proposition}

\newtheorem{corollary}[theorem]{Corollary}

\def\A{\mathbb{A}}
\def\F{\mathbb{F}}
\def\Z{\mathbb{Z}}
\def\N{\mathbb{N}}
\def\M{\mathbb{M}}

\numberwithin{equation}{section}
\numberwithin{theorem}{section}

\begin{document}

\title[Euler's totient function ]{On Euler's totient function of polynomials over finite fields}

\author{Xiumei Li}
\address{School of Mathematical Sciences, Qufu Normal University, Qufu Shandong, 273165, China}
\email{lxiumei2013@qfnu.edu.cn}

\author{Min Sha}
\address{\parbox{\linewidth}{School of Mathematical Sciences, South China Normal University, Guangzhou, 510631, China \\ 
Key Laboratory of Applied Mathematics (Putian University), Fujian Province University, Fujian Putian, 351100, China}}
\email{min.sha@m.scnu.edu.cn}



\begin{abstract}
In this paper, we study some typical arithmetic properties of  Euler's totient function of polynomials over finite fields. Especially, we study polynomial analogues of some classical conjectures about Euler's totient function, such as 
Carmichael's conjecture, Sierpi{\'n}ski's conjecture, and Erd\"{o}s' conjecture. 
\end{abstract}

\maketitle

\section{Introduction}
Let $\F_q$ be the finite field of $q$ elements, where $q$ is a power of a prime $p$.
Denote by $\A=\F_q[x]$ the ring of polynomials in one variable $x$ over $\F_q$ and $\M$ the set of all monic polynomials in $\A$.

In number theory, the analogy between the rational integers $\Z$ and $\A$ yields many ideas and problems for research. 
In this paper, we want to study the polynomial analogue of the classical Euler totient function. 

In the integer case, we recall that for any positive integer $n$, Euler's totient function $\varphi(n)$ is defined to be the number of positive integers $k$ not greater than $n$ and relatively prime to $n$. 

For any non-constant polynomial $f\in \A$, set $|f|=q^{\deg f}$ and denote by $\A/f\A$ the
residue class ring of $\A$ modulo $f$ and by $(\A/f\A)^*$ its unit group.
Let $\Phi(f)$ be the number of elements in the group $(\A/f\A)^*$, that is, $\Phi(f)= |(\A / f\A)^*|$,
which is the so-called \emph{Euler totient function} of $\A$.
It is well-known that (see, for instance, \cite[Proposition 1.7]{Rosen})
\begin{equation}  \label{eq:Euler}
\Phi(f) = |f| \prod_{P\mid f} \big( 1- 1/|P| \big),
\end{equation}
where $P$ runs through all monic irreducible divisors of $f$ in $\A$. 

Studying polynomial analogues of some classical conjectures about Euler's totient function has attracted attention recently. Here we take Lehmer's totient problem for an example. 
In the integer case, Lehmer's totient problem asks whether there is any composite number $n$ such that Euler's totient function $\varphi(n)$ divides $n-1$. This is still an open problem. 
Recently, Ji and Qin \cite{JQ} completely solved Lehmer's totient problem over $\A$. 

In this paper, we want to study some typical arithmetic properties of the functions $\Phi(f)$. 
Especially, we want to study polynomial analogues of some classical conjectures about Euler's totient function, such as 
Carmichael's conjecture, Sierpi{\'n}ski's conjecture, and Erd\"{o}s' conjecture. 
We will describe these conjectures and state our main results in relevant sections. 
In addition, we also show that the natural density of the values of the function $\Phi$ is zero.

\section{Preliminaries}

In this section, we collect some results which are used later on. 

For our purpose, we need to use Zsigmondy's theorem, which was discovered by Zsigmondy \cite{Zsig} in 1892
and independently rediscovered by Birkhoff and Vandier \cite{BV} in 1904.
The special case where $b=1$ was discovered earlier by Bang \cite{Bang} in 1886.

\begin{lemma}[Zsigmondy's theorem]
Let $a,b\in \N$ such that $\gcd(a,b)=1$ and $n\in \N, n >1 $. Then, there exists a prime divisor of $a^n-b^n$
that does not divide $a^k-b^k$ for all $k \in \{1,2,\cdots,n-1\}$ $($we call it a primitive prime divisor$)$, except exactly in the following cases:
\begin{itemize}
\item[(1)] $2^6-1$,
\item[(2)] $n=2$ and $a+b$ is a power of $2$.
\end{itemize}
\end{lemma}

The following well-known version of Stirling's formula for factorials is due to Robbins \cite{Rob}. 
\begin{lemma}[Stirling's formula]   
\label{lem:Stirling}
For any positive integer $n$, we have 
$$
\sqrt{2\pi n} \big(\frac{n}{e}\big)^n e^{1/(12n+1)} < n! < \sqrt{2\pi n} \big(\frac{n}{e}\big)^n e^{1/(12n)}, 
$$
where $e$ is the base of the natural logarithm. 
\end{lemma}

We also need some estimates about the number of non-negative integer solutions to certain Diophantine inequality.

\begin{lemma}[{\cite{BD}}]\label{lem:N} 
Let $N$ be the number of non-negative integer solutions to the Diophantine inequality $a_1x_1 + a_2x_2 + \ldots + a_kx_k \leq n$, where all of $a_i$
are positive integers. Then
\begin{align*}
\frac{n^k}{k!\prod_{i=1}^{k}a_i}\leq N \leq  \frac{\big(n+a_1+a_2+\cdots+a_k\big)^k}{k!\prod_{i=1}^{k}a_i}.
\end{align*}
\end{lemma}

Moreover, using Lemmas~\ref{lem:Stirling} and \ref{lem:N} we obtain a simple upper bound for a special case. 

\begin{lemma}\label{lem:S}
For any integer $n \ge 1$, let $N(n)$ be the number of non-negative integer solutions to the Diophantine inequality $x_1 + 2x_2 + \cdots + nx_n \leq n$. Then, we have 
\begin{equation*} 
N(n) < 2\big(\frac{e^2}{2}\big)^{\frac{n}{2}}.
\end{equation*}
\end{lemma}

\begin{proof}
 By Lemmas~\ref{lem:Stirling} and \ref{lem:N}, we obtain

\begin{equation*}
N(n) \leq \frac{\big(\frac{n^2+3n}{2}\big)^n}{n!n!}< \frac{1}{2\pi n}\big(1+\frac{3}{n}\big)^n \big(\frac{e^2}{2}\big)^n.
\end{equation*}

Recalling the well-known inequality $(1+1/x)^x < e$ for any $x>0$, we have
$$
\frac{1}{2\pi n}\big(1+\frac{3}{n}\big)^n < \frac{e^3}{2\pi n} < 1, \ \textit{for}\ n\geq 4.
$$
 Hence, if $n\geq 4$, then $N(n)<\big(\frac{e^2}{2}\big)^n$. If $n=1,2,3$, it's easy to check that
$$
\frac{\big(\frac{n^2+3n}{2}\big)^n}{n!n!} <\big(\frac{e^2}{2}\big)^n,
$$
 which implies that $N(n)<\big(\frac{e^2}{2}\big)^n$.
Thus, we obtain a primary estimate of $N(n)$, that is, for any positive integer $n$,
\begin{equation}
\label{eq:N-1}
N(n) < \big(\frac{e^2}{2}\big)^n.
\end{equation}

Next, in order to achieve our desired estimate of $N(n)$, we need to classify the non-negative integer solutions of the Diophantine inequality in the lemma.

For simplicity, we denote $k = \lfloor\frac{n}{2}\rfloor$.
Note that for any non-negative integer solution $(x_1,x_2,\ldots,x_n)$
of the inequality $x_1 + 2x_2 + \ldots + nx_n \leq n$, there is at most one $j$ with $k+1 \leq j \leq n$
 such that $x_j = 1$. So, we can classify them by whether there exists some $k+1 \leq j \leq n$ such that $x_j = 1$.

 If $x_n=1$, we must have $(x_1,x_2,\ldots,x_{n-1}) = (0,0,\ldots, 0)$.
If $x_j=1$ with $ k+1 \leq j < n $, then $(x_1,x_2,\ldots,x_{n-j})$ is a solution of the inequality $x_1 + 2x_2 + \cdots + (n-j)x_{n-j} \leq n-j$.
If $x_j=0$ for any $ k+1 \leq j < n $, then $(x_1,x_2,\ldots,x_{k})$ satisfies the inequality $x_1 + 2x_2 + \cdots + kx_{k} \leq n$.
Denote by $T(k)$ the number of solutions of $x_1 + 2x_2 + \cdots + kx_{k} \leq n$, and put $N(0)=1$. 
So, we have
\begin{align}\label{eq:N}
N(n) = 1 + \sum_{j=k+1}^{n-1}N(n-j) + T(k)
 =\sum_{i=0}^{n-(k+1)}N(i) + T(k).
\end{align}

For $\sum_{i=0}^{n-(k+1)}N(i)$, by \eqref{eq:N-1}, we have
\begin{align}\label{eq:N-2}
\sum_{i=0}^{n-(k+1)}N(i)< \sum_{i=0}^{n-(k+1)}\big(\frac{e^2}{2}\big)^i
=\frac{\big(\frac{e^2}{2}\big)^{n-k}-1}{\frac{e^2}{2}-1}<\big(\frac{e^2}{2}\big)^{n/2}.
\end{align}

For $T(k)$, by Lemmas \ref{lem:N} and \ref{lem:Stirling} again, we obtain
\begin{equation}\label{eq:T}
    T(k) \leq \frac{\big(n+\frac{k(k+1)}{2}\big)^{k}}{k!k!}< \frac{1}{2\pi k}\big(1+\frac{1}{k} + \frac{2n}{k^2}\big)^k \big(\frac{e^2}{2}\big)^k.
\end{equation}

Now we discuss case by case according to the parity of $n$.

When $n$ is even and $k=n/2$, we have
$$
\frac{1}{2\pi k}\big(1+\frac{1}{k} + \frac{2n}{k^2}\big)^k = \frac{1}{2\pi k}\big(1+\frac{5}{k}\big)^k < \frac{e^5}{2\pi k} < 1,  \ \textit{for}\  k>23.
$$
So, if $k>23$, by \eqref{eq:T}, we have $T(k) < \big(\frac{e^2}{2}\big)^{n/2}$.
If $k\leq 23$, using PARI/GP, we can check that
$$
\frac{\big(n+\frac{k(k+1)}{2}\big)^{k}}{k!k!} < \big(\frac{e^2}{2}\big)^{n/2},
$$
which also follows that $T(k) < \big(\frac{e^2}{2}\big)^{n/2}$.

When $n$ is odd and $k=(n-1)/2$, we have
$$
   \frac{1}{2\pi k}\big(1+\frac{1}{k} + \frac{2n}{k^2}\big)^k < \frac{1}{2\pi k}\big(1+\frac{6}{k}\big)^k < \frac{e^6}{2\pi k} < 1,   \ \textit{for}\  k>64.
$$
So, if $k>64$, by \eqref{eq:T}, we have $T(k) < \big(\frac{e^2}{2}\big)^k<\big(\frac{e^2}{2}\big)^{n/2}$. If $k\leq 64$, using PARI/GP, we can check that
$$
\frac{\big(n+\frac{k(k+1)}{2}\big)^{k}}{k!k!} < \big(\frac{e^2}{2}\big)^{n/2},
$$
which also follows that $T(k) < \big(\frac{e^2}{2}\big)^{n/2}$.

For the above discussion, we obtain that for any positive integer $n$, the inequality
 \begin{equation}\label{eq:N-3}
 T(k) < \big(\frac{e^2}{2}\big)^{n/2}
 \end{equation}
 always holds.

Finally, by \eqref{eq:N} and \eqref{eq:N-2} and \eqref{eq:N-3}, we have
\begin{align*}
  N(n) < \big(\frac{e^2}{2}\big)^{\frac{n}{2}} + \big(\frac{e^2}{2}\big)^{\frac{n}{2}} = 2\big(\frac{e^2}{2}\big)^{\frac{n}{2}}.
\end{align*}
This completes the proof.
\end{proof}

For any integer $d \ge 1$, let $\pi_q(d)$ be the number of monic irreducible
  polynomials of degree $d$ in $\A$.
 It is well-known that (see, for example, \cite[Corollary of Proposition 2.1]{Rosen})
\begin{equation}  \label{eq:Nq}
\pi_q(d) = \frac{1}{d} \sum_{j \mid d} \mu(j)q^{\frac{d}{j}},
\end{equation}
where $j$ runs over all the positive divisors of $d$, and $\mu(j)$ is the M\"{o}bius function.

Recall that $q$ is a power of a prime $p$, say $q = p^s, s\geq 1$.

\begin{lemma}
Assume that $q \neq 2$. Then, for any integer $d \geq 1$, either $p \mid \pi_q(d)$, or $4 \mid \pi_q(d)$.
\end{lemma}
\begin{proof}
It's easy to check that $p \mid \pi_q(d)$ when $d$ is not square-free or $p\nmid d$ or $s > 1$. So, we need only consider
the case: $d$ is square-free and $p\mid d$ and $q = p$ is odd prime.

When $2\nmid d$, write $d = pp_1\ldots p_k$, where $p,p_1,\ldots, p_k$ are distinct odd prime. By \eqref{eq:Nq}, we have
\begin{align*}
   \pi_p(d) &= \frac{1}{p_1\ldots p_k}\sum_{j|d}\mu(j)p^{\frac{d}{j}-1}.
\end{align*}
Note that
\begin{align*}
\sum_{j|d}\mu(j)p^{\frac{d}{j}-1}\equiv \sum_{j|d}\mu(j)\pmod 8\equiv 0 \pmod 8,
\end{align*}
where we use the fact that $p^{2} \equiv 1 \pmod 8$ and $\sum_{j|d}\mu(j) = 0$ when $d > 1$.
Thus, we have $\pi_p(d)\equiv 0 \pmod 8$.

When $2\mid d$, write $d = 2pp_1\ldots p_k$, where $p,p_1,\ldots, p_k$ are distinct odd prime. Hence, we have
\begin{align*}
   \pi_p(d) &= \frac{1}{2p_1\ldots p_k}\sum_{j|d}\mu(j)p^{\frac{d}{j}-1}\\
   &= \frac{1}{2p_1\ldots p_k}\Big(\sum_{2|j|d}\mu(j)p^{\frac{d}{j}-1} + \sum_{2\nmid j|d}\mu(j)p^{\frac{d}{j}-1}\Big).
\end{align*}
Note that
\begin{align*}
\sum_{2|j|d}\mu(j)p^{\frac{d}{j}-1}& = \sum_{j|pp_1\ldots p_k}\mu(2j)p^{\frac{d}{2j}-1} \equiv -\sum_{j|pp_1\ldots p_k}\mu(j)\pmod 8 \\
&\equiv 0 \pmod 8
\end{align*}
and
\begin{align*}
\sum_{2\nmid j|d}\mu(j)p^{\frac{d}{j}-1} &= \sum_{j|pp_1\ldots p_k}\mu(j)p^{\frac{d}{j}-1} \equiv p\sum_{j|pp_1\ldots p_k}\mu(j)\pmod 8\\
&\equiv 0 \pmod 8,
\end{align*}
which implies that $ \pi_p(d) \equiv 0 \pmod 4 $.

For the above discussion, we have proved that $ \pi_p(d) \equiv 0 \pmod 4 $ when $d$ is square-free and $p\mid d$ and $q = p$ is odd prime.
\end{proof}

\section{Collision of totient values}
\label{sec:cond-E}

In this section, we want to determine under which condition two polynomials have the same value of Euler's totient function.

We first introduce some notations. 
For any non-constant polynomial $f \in \A$ and any positive integer $d$,
let $m_{d}(f)$ be the number of monic irreducible polynomials in $\A$ of degree $d$ and dividing $f$.
Clearly, we have $0 \le m_d(f) \le \pi_q(d)$,
and moreover $m_d(f) = 0$ if $d$ is greater than the maximal degree of irreducible divisors of $f$.

We associate the following set to $f$:
\begin{equation*}\label{eq:S}
S(f)=\{(d,m_{d}(f)): \, d =1,2, 3, \ldots \}.
\end{equation*}
Note that there are at most finitely many positive integers $d$ with $m_d(f) > 0$.

By definition and using \eqref{eq:Euler}, we have
\begin{equation}   \label{eq:Phi}
\Phi(f) = q^{\deg f - \sum_{d=1}^{\infty}dm_d(f)} \prod_{d=1}^{\infty} \big( q^d - 1 \big)^{m_d(f)}.
\end{equation}
Then, it is easy to see that for any non-constant polynomials $f, g \in \M$, if
$\deg f = \deg g $ and $S(f) = S(g)$, then $\Phi(f) = \Phi(g)$.
In the sequel, we want to determine a necessary and sufficient condition when $\Phi(f) = \Phi(g)$.

Now we present and prove the main result of this section.

\begin{theorem}\label{thm:same-E}
For any non-constant polynomials $f, g \in \M$, $\Phi (f) = \Phi (g)$ if and only if one of the following
conditions holds:
\begin{itemize}
\item[(1)] when $q\neq 2, 3$, $\deg f = \deg g$ and $S(f) = S(g)$.

\item[(2)] when $q = 3$, $m_d(f) = m_d(g)$ for any  $d \ge 3$,
   $m_1(f) + 3m_2(f) = m_1(g) + 3m_2(g)$, and $\deg f + m_2(f) = \deg g + m_2(g)$.

\item[(3)] when $q = 2$, $m_d(f) = m_d(g)$ for any $d \ge 2$,
and $\deg f -  m_1(f) = \deg g - m_1(g)$.
\end{itemize}
\end{theorem}

\begin{proof}
By \eqref{eq:Phi}, we know that $\Phi (f) = \Phi (g)$ if and only if
$$
q^{\deg f - \sum_{d=1}^{\infty}dm_d(f)} \prod_{d=1}^{\infty} ( q^d -1 )^{m_d(f)}
= q^{\deg g - \sum_{d=1}^{\infty}dm_d(g)} \prod_{d=1}^{\infty} ( q^d - 1 )^{m_d(g)},
$$
that is, if and only if
\begin{equation} \label{eq:q}
\deg f - \sum_{d=1}^{\infty}dm_d(f) = \deg g - \sum_{d=1}^{\infty}dm_d(g)
\end{equation}
and
\begin{equation} \label{eq:q-1}
\prod_{d=1}^{\infty} ( q^d - 1)^{m_d(f)} = \prod_{d=1}^{\infty} ( q^d - 1)^{m_d(g)}.
\end{equation}

For sufficiency, it is easy to see that if one of the conditions (1), (2) and (3) holds,
then \eqref{eq:q} and \eqref{eq:q-1} both hold, and so $\Phi (f) = \Phi (g)$.

So, it remains to prove the necessity.
Assume that $\Phi (f) = \Phi (g)$. Then, both \eqref{eq:q} and \eqref{eq:q-1} hold.
We now complete the proof case by case.

(1)
In this case, $q \neq 2, 3$.
Note that there are only finitely many positive integers $d$ with $m_d(f) > 0$.
Then, considering the primitive prime divisors in the sequence $(q^n - 1)_{n \ge 1}$ and using Zsigmondy's theorem,
from \eqref{eq:q-1} we deduce that  $m_d(f)=m_d(g)$ for any $d \ge 3$,
and so \eqref{eq:q-1} becomes
\begin{equation} \label{eq:q-12}
\big( q-1 \big)^{m_1(f)}\big( q^2-1 \big)^{m_2(f)} = \big( q-1 \big)^{m_1(g)}\big( q^2-1 \big)^{m_2(g)}.
\end{equation}

When $q+1$ is not a power of $2$, by Zsigmondy's theorem we know that $q^{2}-1$ has a prime divisor not dividing $q-1$,
and thus we must have $m_2(f)=m_2(g)$, and then $m_1(f)=m_1(g)$.
Hence, we have $S(f) = S(g)$.

When $q+1$ is a power of $2$, write $q+1=2^s$. Note that $s\geq 3$, because $q \ne 3$.
The equation \eqref{eq:q-12} becomes
$$
\big( \frac{q-1}{2} \big)^{m_1(f)+m_2(f)} 2^{m_1(f) + (s+1)m_2(f)}
= \big( \frac{q-1}{2} \big)^{m_1(g)+m_2(g)} 2^{m_1(g) + (s+1)m_2(g)}.
$$
Since $2$ and $\frac{q-1}{2}$ are coprime integers and $\frac{q-1}{2} \ge 3$,
we have that $m_1(f) + m_2(f) = m_1(g) + m_2(g)$ and $m_1(f) + (s+1)m_2(f) = m_1(g) + (s+1)m_2(g)$,
which gives $m_1(f) = m_1(g)$ and $m_2(f) = m_2(g)$.
Hence, we obtain $S(f) = S(g)$.

Therefore, in this case we always have $S(f) = S(g)$.
Then, in view of \eqref{eq:q} we have $\deg f = \deg g$.
This completes the proof of the case (1).

(2)
In this case, $q = 3$.
As in the proof of (1), we deduce that $m_d(f) = m_d(g)$ for any $d \ge 3$  and \eqref{eq:q-12} holds.
Since $q = 3$, \eqref{eq:q-12} gives
$$
m_1(f) + 3m_2(f) = m_1(g) + 3m_2(g).
$$
By \eqref{eq:q}, we obtain $\deg f - m_1(f) - 2m_2(f) = \deg g - m_1(g) - 2m_2(g)$, and so
$$
\deg f + m_2(f) = \deg g + m_2(g).
$$
This completes the proof of the case (2).

(3)
In this case, $q = 2$.
As before, considering the primitive prime divisors in the sequence $(2^n - 1)_{n \ge 1}$ and using Zsigmondy's theorem,
from \eqref{eq:q-1} we obtain  $m_d(f)=m_d(g)$ for any $d \ge 7$,
and so \eqref{eq:q-1} becomes
\begin{equation} \label{eq:q-13}
\prod_{d =1}^{6} \big( 2^d - 1 \big)^{m_d(f)} = \prod_{d=1}^{6} \big( 2^d - 1 \big)^{m_d(g)},
\end{equation}
which further becomes
\begin{align*}
3^{m_2(f)+m_4(f) + 2m_6(f)} \cdot 5^{m_4(f)} \cdot 7^{m_3(f) + m_6(f)} \cdot 31^{m_5(f)}  \\
= 3^{m_2(g)+m_4(g) + 2m_6(g)} \cdot 5^{m_4(g)} \cdot 7^{m_3(g) + m_6(g)} \cdot 31^{m_5(g)}.
\end{align*}
Hence, we obtain that  $m_4(f) = m_4(g)$, $m_5(f) = m_5(g)$,   $m_2(f) + 2 m_6(f) = m_2(g) + 2m_6(g)$,
and $ m_3(f) +  m_6(f) = m_3(g) + m_6(g)$.

Recalling $\pi_q(d)$ in \eqref{eq:Nq}, since $0 \le m_2(f) \le \pi_2(2) = 1$, we have $-1 \le m_2(f) - m_2(g) \le 1$.
Note that we have shown $m_2(f) - m_2(g) = 2 (m_6(f) - m_6(g))$.
So, we must have $m_2(f) = m_2(g)$, and then $m_6(f) = m_6(g)$, and also $m_3(f) = m_3(g)$.

Finally, combining \eqref{eq:q} with the above discussion, we obtain
$$
\deg f - m_1(f) = \deg g - m_1(g).
$$
This completes the proof.
\end{proof}

\section{Carmichael's conjecture}
\label{sec:Car-E}

In the integer case, Carmichael's conjecture \cite{Carm1907,Carm1922} asserts that for any positive number $n$, either $|\phi^{-1}(n)| = 0$ or $|\phi^{-1}(n)| \geq 2$,
where $\phi$ is the classical Euler totient function and $\phi^{-1}(n)$ is the inverse image of $n$. This conjecture is still an open problem, and the current best result by Ford \cite[Theorem 6]{Ford1998} asserts that  a counterexample to 
Carmichael's conjecture must exceed $10^{10^{10}}$.

In this section, we want to  study the polynomial analogue of Carmichael's conjecture.

For any non-empty subset $B$ of non-zero polynomials in $\A$, define $\Phi(B)$ to be the set of values $\Phi(f), f \in B$.
By \eqref{eq:Phi},  it is easy to determine  the set  $\Phi(\A)$.

\begin{proposition}
\label{pro:val-E}
\begin{align*}
   \Phi(\A) = & \Big\{q^{j}\prod_{d=1}^{k}\big( q^d-1\big)^{m_{d}}:~ j=j_1 +  \ldots + k j_k \text{ for some $k \ge 1$ and} \\
   & \text{non-negative integers $j_1, ..., j_k, m_1, \ldots, m_k$ satisfying $m_k \ge 1$,} \\
   &  \text{$m_{d} \leq \pi_q(d)$, and $j_d=0$ if $m_d=0$, for each $1\le d \le k$}\Big\}.
\end{align*}
\end{proposition}

For any positive integer $ n\in \Phi(\A)$, denote by $\Phi^{-1}(n)$ the inverse image of $n$.
We now want to compute the cardinality of the intersection  $\Phi^{-1}(n)\cap \M$ for each $n \in \Phi(\A)$.

Recall that $\M$ is the set of monic polynomials in $\A$.
Clearly, $\Phi(\A) = \Phi(\M)$.
Informally, monic polynomials are analogues of positive integers.

\begin{theorem}\label{thm:num:q}
Assume that $q \neq 2, 3$. Then, for any $n = q^{j}\prod_{d=1}^{k}\big( q^d-1\big)^{m_{d}} \in \Phi(\A)$
as in Proposition~\ref{pro:val-E}, we have
\begin{equation*}
 \Big|\Phi^{-1}(n)\cap \M\Big|  =
\sum_{j_1 + \ldots + k j_k = j} \prod_{d=1}^{k} C_d\binom{\pi_q(d)}{m_d},
\end{equation*}
 where $(j_1, \ldots, j_k)$ runs over all of non-negative integer solutions of the equation $j_1 + \ldots + k j_k = j$ and
 \begin{equation*}
\begin{split}
C_d
& = \left\{\begin{array}{ll}
 0 \quad  \textit{if $m_d =0, j_d > 0$},\\
 1 \quad  \textit{if $m_d =0, j_d = 0$},\\
 \binom{j_d + m_d - 1}{m_d - 1} \quad \textit{if $m_d >0$}.
\end{array}
\right.
\end{split}
\end{equation*}
\end{theorem}

\begin{proof}
Note that for any $m_d \ge 1$ and $j_d \ge 0$, the Diophantine equation
$x_1 + x_2 + \ldots + x_{m_d} = j_d$ has $\binom{j_d + m_d - 1}{m_d - 1}$ non-negative integer solutions.
Hence, using Theorem~\ref{thm:same-E} (1), we first fix $m_d$ monic irreducible factors of degree $d$ for each $1 \le d \le k$,
and then we  have
\begin{align*}
 \Big|\Phi^{-1}(n)\cap \M\Big| & = \Big(\prod_{d=1}^{k} \binom{\pi_q(d)}{m_d} \Big) \cdot
 \sum_{j_1 + \ldots + k j_k = j} \prod_{\substack{d=1 \\ m_d \ge 1}}^{k} \binom{j_d + m_d - 1}{m_d - 1} \\
 & = \sum_{j_1 + \ldots + k j_k = j} \prod_{d=1}^{k} C_d\binom{\pi_q(d)}{m_d},
\end{align*}
 where $j_1, \ldots, j_k$ are as described in the theorem.
\end{proof}






When $q=3$, for any $n = 3^{j}\prod_{d=1}^{k}\big( 3^d-1\big)^{m_{d}} \in \Phi(\A)$ with $m_k \ge 1$ as in Proposition~\ref{pro:val-E},
we have
\begin{equation}  \label{eq:n3}
n = 3^{j}\prod_{d=1}^{k}\big( 3^d-1\big)^{m_{d}} = 2^i 3^{j}\prod_{d=3}^{k}\big( 3^d-1\big)^{m_{d}}
\end{equation}
 (when $k < 3$, the part $\prod_{d=3}^{k}\big( 3^d-1\big)^{m_{d}}$ equals to 1),
where $i = m_1 + 3m_2$. Moreover, since $m_1 \le \pi_3(1)=3$ and $m_2 \le \pi_3(2) =3$, we have $i \le 12$.

\begin{theorem}\label{thm:num:3}
Assume that $q = 3$. Then, for any $n =  2^i 3^{j}\prod_{d=3}^{k}\big( 3^d-1\big)^{m_{d}} \in \Phi(\A)$ as in \eqref{eq:n3}, we have
\begin{equation*}
 \Big|\Phi^{-1}(n)\cap \M\Big|  =
\sum_{m_1+3m_2=i}\sum_{j_1 +  \ldots + k j_k = j} \prod_{d=1}^{k} C_d\binom{\pi_3(d)}{m_d},
\end{equation*}
where $C_d$ has been defined in Theorem \ref{thm:num:q}.
\end{theorem}

\begin{proof}
As in the proof of Theorem \ref{thm:num:q}, using Theorem~\ref{thm:same-E} (2) we obtain
\begin{align*}
& \Big|\Phi^{-1}(n)\cap \M\Big|
 =  \Big|\Phi^{-1}(2^i 3^{j}\prod_{d=3}^{k}\big( 3^d-1\big)^{m_{d}})\cap \M\Big| \\
&\quad =\sum_{m_1+3m_2=i} \Big( \big(\prod_{d=1}^{k} \binom{\pi_3(d)}{m_d} \big) \cdot
 \sum_{j_1 + \ldots + k j_k = j} \prod_{\substack{d=1 \\ m_d \ge 1}}^{k} \binom{j_d + m_d - 1}{m_d - 1} \Big)\\
&\quad =\sum_{m_1+3m_2=i}\sum_{j_1  + \ldots + k j_k = j}\prod_{d=1}^{k} C_d\binom{\pi_3(d)}{m_d}.
\end{align*}
This completes the proof.
\end{proof}



When $q=2$, we have $(2-1)^{m_1} = 1$ for any $m_1 \ge 0$.
Similar as the above, applying Theorem~\ref{thm:same-E} (3) we directly obtain:

\begin{theorem}\label{thm:num:2}
Assume that $q = 2$. Then, for any $n =2^{j}\prod_{d=1}^{k}\big( 2^d-1\big)^{m_{d}} \in \Phi(\A)$ as in Proposition~\ref{pro:val-E}, we have
\begin{equation*}
 \Big|\Phi^{-1}(n)\cap \M\Big|  =
\sum_{m_1=0}^{2} \sum_{j_1 +  \ldots + k j_k = j} \prod_{d=1}^{k} C_d\binom{\pi_2(d)}{m_d},
\end{equation*}
where $C_d$ has been defined in Theorem \ref{thm:num:q}.
\end{theorem}



Now we are ready to present and prove the main results of this section.

By Theorem \ref{thm:num:q} and Theorem \ref{thm:num:3}, we can obtain the following Theorem \ref{thm:Phi-1} directly, which asserts that: when $q \neq 2$, 
for any positive integer $n$, either $\Big|\Phi^{-1}(n)\cap \M\Big| = 0$ or $\Big|\Phi^{-1}(n)\cap \M\Big| \geq 1$.

\begin{theorem}\label{thm:Phi-1}

For any $n = q^{j}\prod_{d=1}^{k}\big( q^d-1\big)^{m_{d}} \in \Phi(\A)$
as in Proposition~\ref{pro:val-E}, $\Big|\Phi^{-1}(n)\cap \M\Big| = 1$ if and only if one of the following
conditions holds:\begin{itemize}
\item[(1)] when $q\neq 2, 3,  n = \prod_{d=1}^{k}\big( q^d-1\big)^{m_{d}}$, where $m_{d} = 0$ or $ \pi_q(d), d =1,2,\ldots, k, k\in \Z_{\geq 1 }$.

\item[(2)] when $q = 3,  n = \prod_{d=1}^{k}\big( q^d-1\big)^{m_{d}}$, where $m_{d} = 0$ or $ \pi_q(d), d =1,2,\ldots, k,k\in \Z_{\geq 2 }$ and $m_1=m_2$.

\end{itemize}

\end{theorem}

When $q=2$, similar as the above, applying Theorem~\ref{thm:same-E} (3) and Theorem~\ref{thm:num:2} we obtain:

\begin{theorem}\label{thm:num:2-2}

Assume that $q = 2$. Then, for any $n = 2^{j}\prod_{d=1}^{k}\big( 2^d-1\big)^{m_{d}} \in \Phi(\A)$
as in Proposition~\ref{pro:val-E}, we have
$$ \Big|\Phi^{-1}(n)\cap \M\Big|  \geq 3,$$
 where the equality holds only when $n = 1$.

\end{theorem}

\begin{proof}
By Proposition~\ref{pro:val-E} and Theorem~\ref{thm:num:2}, we consider the following cases by $n$.
If $n = 1$, then $ \Big|\Phi^{-1}(n)\cap \M\Big|  =
\sum_{m_1=1}^{2}\binom{\pi_2(1)}{m_1} = 2 + 1 = 3$.

If $n > 1$ is odd, then $ j = 0$ and $m_d \geq 1$ for some $d \geq 2$. So, we have \begin{align*}
 \Big|\Phi^{-1}(n)\cap \M\Big|  = \Big(\binom{2}{0} + \binom{2}{1} + \binom{2}{2}\Big)\prod_{d=2}^{k}\binom{\pi_2(d)}{m_d}
 = 4\prod_{d=2}^{k}\binom{\pi_q(d)}{m_d}.
\end{align*}

 If $n$ is even, then $j \geq 1$. So, we have \begin{align*}
 \Big|\Phi^{-1}(n)\cap \M\Big|   \geq \sum_{m_1=1}^{2} C_1\binom{\pi_2(1)}{m_1} \prod_{d=2}^{k} \binom{\pi_2(d)}{m_d}
\geq 4\prod_{d=2}^{k} \binom{\pi_q(d)}{m_d},
\end{align*}
where we use the fact that $j_1 = j, j_2 = j_3 = \cdots = j_k = 0$ is one of solutions of the equation $j_1 + 2j_2 + \cdots + kj_k = j$.

Thus, in any case, we always have
$
\Big|\Phi^{-1}(n)\cap \M\Big| \geq 3$
and the equality holds only when $n = 1$.
\end{proof}

Theorem \ref{thm:num:2-2} asserts that: when $q = 2$,
for any positive integer $n$, either $\Big|\Phi^{-1}(n)\cap \M\Big| = 0$ or $\Big|\Phi^{-1}(n)\cap \M\Big| \geq 3$.

\section{Sierpi{\'n}ski's conjecture}
\label{sec:Sie-E}

 In the integer case, Sierpi{\'n}ski's conjecture \cite{Sch1956} states that for every integer $l \geq 2$, there exists an integer $n$ such that the equation $\phi (t) = n$ has exactly $l$ solutions. This conjecture was proved by Ford \cite{Ford1999}.

 In this section, we want to  study the polynomial analogue of Sierpi{\'n}ski's conjecture.

\begin{theorem}
Assume that $q = 2$. Then, for every integer $l \geq 3$, there exists an integer $n$ such that $\Big|\Phi^{-1}(n)\cap \M\Big| = l$.
\end{theorem}

\begin{proof}
Since $\Big|\Phi^{-1}(1)\cap \M\Big| = 3$, we only consider the case $l > 3$. Define $n = 2^{l-3}$, so, by Theorem \ref{thm:num:2}, we have
$$\Big|\Phi^{-1}(n)\cap \M\Big| = \sum_{m_1=1}^{2} C_1\binom{\pi_2(1)}{m_1} = 2 + \binom{l - 3 + 2 -1}{2 - 1 } = l.$$
\end{proof}

\begin{theorem}
Assume that $q \neq 2$. Then, for any integer $l \geq 1$, there exists an integer $n$ such that
$\Big|\Phi^{-1}(n)\cap \M\Big| = q^{l}$.
\end{theorem}
\begin{proof}
Define $ n = q^{l}\prod_{d=1}^{l}\big( q^d-1)^{\pi_{q}(d)}$. So, by Theorem~\ref{thm:num:q} and Theorem~\ref{thm:num:3}, we have
\begin{align*}
 \Big|\Phi^{-1}\Big(n\Big)\cap \M\Big| & =
 \sum_{j_1 + \ldots + l j_l = l} \prod_{\substack{d=1 }}^{l} \binom{j_d + \pi_q(d) - 1}{\pi_q(d) - 1} = q^l,
\end{align*}
where we use the fact that $ \sum_{j_1 + \ldots + l j_l = l} \prod_{\substack{d=1 }}^{l} \binom{j_d + \pi_q(d) - 1}{\pi_q(d) - 1} $
is the number of all monic polynomials of degree $l$ in $\A$.
\end{proof}

\begin{theorem}
Assume that $q \neq 2$. Then, for any integer $l \geq 0$, there exists an integer $n$ such that
$\Big|\Phi^{-1}(n)\cap \M\Big| = \binom{q}{2}(l+1)$.
\end{theorem}

\begin{proof}
Define $ n = q^l(q-1)^2$. So, by Theorem~\ref{thm:num:q} and Theorem~\ref{thm:num:3}, we have
$$\Big|\Phi^{-1}(n)\cap \M\Big| = \binom{l + 2 - 1}{2 - 1}\binom{\pi_q(1)}{2} = \binom{q}{2}(l+1),$$
where we use the fact that $ \pi_q(1) = q$.
\end{proof}

The following theorem tells us that: when $q \neq 2, 3$, for any integer $l\in\Big((1,q)\cup(q,\frac{q(q-1)}{2})\Big)\cap\Z$, the equation $\Big|\Phi^{-1}(n)\cap \M| = l $ has no solution.

\begin{theorem}

Assume that $q \neq 2, 3$. Then, for any $n = q^{j}\prod_{d=1}^{k}\big( q^d-1\big)^{m_{d}} \in \Phi(\A)$
as in Proposition~\ref{pro:val-E}, we have
$$ \Big|\Phi^{-1}(n)\cap \M| = 1\ \text{ or }\  |\Phi^{-1}(n)\cap \M\Big| = q\ \text{ or } \ \Big|\Phi^{-1}(n)\cap \M\Big| \geq \binom{q}{2}.$$
\end{theorem}

\begin{proof}

 By Theorem \ref{thm:num:q}, we have
 $$\Big|\Phi^{-1}(n)\cap \M\Big| = \Big(\prod_{d=1}^{k} \binom{\pi_q(d)}{m_d} \Big) \cdot
 \sum_{j_1 + \ldots + k j_k = j} \prod_{\substack{d=1 \\ m_d \ge 1}}^{k} \binom{j_d + m_d - 1}{m_d - 1}.$$

If $1 \leq m_d < \pi_q(d)$ for some $2\leq d \leq k$ or $2 \leq m_1 < \pi_q(1) - 1$,
then
 $$ \Big|\Phi^{-1}(n)\cap \M\Big|\geq \binom{\pi_q(d)}{m_d}\geq \pi_q(2) = \binom{q}{2}. $$
where we use the fact that $ \binom{\pi_q(1)}{m_1} \geq \binom{\pi_q(1)}{2} = \pi_q(2)$.

If $j>0$ and the equation has a solution $(j_1,\ldots,j_k)$ as Proposition~\ref{pro:val-E}, which satisfies $j_d >0$ for some $2\leq d \leq k$, then, we have
 \begin{equation*}
\Big|\Phi^{-1}(n)\cap \M\Big| \geq \binom{j_d + m_d - 1}{m_d - 1}\cdot\binom{\pi_q(d)}{m_d} \geq \pi_q(d)\geq \binom{q}{2}.
\end{equation*}

If $ j > 0 $ and the equation has a solution $(j_1,\ldots,j_k)$ as Proposition~\ref{pro:val-E}, which satisfies $j_1 >0$ and $ m_1 = \pi_q(1) - 1$, then
 \begin{equation*}
\Big|\Phi^{-1}(n)\cap \M\Big| \geq \binom{j_1 + m_1 - 1}{m_1 - 1}\cdot\binom{\pi_q(1)}{m_1} \geq q(q-1)> \binom{q}{2} .
\end{equation*}

If $ j \geq 0 $ and the equation has only one solution $(j,0,\ldots,0)$ as Proposition~\ref{pro:val-E} and $ m_d = 0$ or $ \pi_q(d)$ for any $2 \leq d \leq k$ and $ m_1 = 1$, then
 \begin{equation*}
\Big|\Phi^{-1}(n)\cap \M\Big| = \binom{j + m_1 - 1}{m_1 - 1}\cdot\binom{\pi_q(1)}{m_1} = q.
\end{equation*}

If $j =0$ and $ m_d = 0$ or $ \pi_q(d)$ for any $2 \leq d \leq k$ and $ m_1 = \pi_q(1) - 1$, we have
\begin{equation*}
\Big|\Phi^{-1}(n)\cap \M\Big| = \binom{\pi_q(1)}{m_1} = q.
\end{equation*}

If $j =0$ and $ m_d = 0$ or $ \pi_q(d)$ for any $1 \leq d \leq k$, we have
\begin{equation*}
\Big|\Phi^{-1}(n)\cap \M\Big| = 1
\end{equation*}

Therefore, in any case, we always have
$$ \Big|\Phi^{-1}(n)\cap \M\Big| = 1\ \text{ or }\  \Big|\Phi^{-1}(n)\cap \M\Big| = q\ \text{ or } \ \Big|\Phi^{-1}(n)\cap \M\Big| \geq \binom{q}{2}.$$
\end{proof}

\section{Erd\"{o}s' conjecture}
\label{sec:Erd-E}

 In the integer case P. Erd\"{o}s \cite{Erd1959, Sch-Sie1958} conjectured that Euler's totient function and the sum-of-divisors function have infinitely many
 common values. This conjecture was proved by K. Ford, F. Luca and C. Pomerance \cite{FLP2010}.

 In this section, we study the polynomial analogue of Erd\"{o}s' conjecture.

For any non-constant polynomial $f\in \A$, the sum-of-divisors function of $f$ is defined by $\sigma(f)=\sum_{g\mid f}|g|$, where the sum is over all monic divisors of $f$.
For any integer $e \ge 1$ and any irreducible polynomial $P$ in $\A$, $P^e \| f$ means that $P^{e} \mid f$ but $P^{e+1} \nmid f$.
It is well-known that (see, for instance, \cite[Proposition 2.4]{Rosen})
\begin{equation}  \label{eq:sigma}
\sigma(f) = \prod_{P^e \| f} \frac{|P|^{e+1}-1}{|P|-1},
\end{equation}
where $P$ runs through all monic irreducible divisors of $f$.

By \eqref{eq:sigma}, we can rewrite $\sigma(f)$ as
\begin{equation}\label{eq:sig3}\sigma(f) = \prod_{d = 1}^{\infty} ( q^d - 1 )^{k_d(f)},\end{equation}
where $\sum_{d = 1}^{\infty}k_d(f) = 0$. Note that there are only finitely many positive integers $d$ with $k_d(g) \neq 0$. Moreover, for any $d \geq 1 $, if $k_d(f) < 0$, then $0 < -k_d(f) \leq \pi_q(d)$. In addition, if $k_d(f) < 0$, there exists some $j$ such that $d|j$ and $k_j(f) >0$.

\begin{theorem} \label{thm:Erd1}

Assume that $q\neq 2, 3$. Then, we have $\Phi( \A ) \cap \sigma( \A ) = \emptyset$.

\end{theorem}

\begin{proof}

 By contradiction, we assume that there exists $f,g\in\A$ such that $\Phi(f) = \sigma(g)$, then by \eqref{eq:Phi} and \eqref{eq:sig3}, we have

\begin{equation}\label{eq:Erd}\prod_{d = 1}^{\infty} ( q^d - 1 )^{m_d(f)} = \prod_{d = 1}^{\infty} ( q^d - 1 )^{k_d(g)}.\end{equation}

As in the proof of Theorem \ref{thm:same-E} (1), considering the primitive prime divisors in the sequence $(q^n - 1)_{n \ge 1}$ and using Zsigmondy's theorem,
we have $k_d(g )= m_d(f)$ for any $d \geq 1$, which implies that $\sum_{d = 1}^{\infty}k_d(g) > 0$. This contradicts with $\sum_{d = 1}^{\infty}k_d(g) = 0$.
So we complete the proof.

\end{proof}

\begin{theorem} \label{thm:Erd3}
Assume that $q = 3$. Then, we have
$$\Phi(\A) \cap \sigma(\A) = \{(3^{d_1}-1) (3^{d_2}-1)| d_1, d_2 \geq 1 \}.$$

\end{theorem}

\begin{proof}
For any $f, g \in \A$, we have $\Phi(f) = \sigma(g)$ if and only if \eqref{eq:Erd} holds. As in the proof of Theorem~ \ref{thm:same-E} (2), this is equivalent to that
$k_d(g) = m_d(f)$ for any $d \ge 3$ and
\begin{equation}\label{eq:Erd-3}
m_1(f)+3m_2(f) = k_1(g) + 3k_2(g).
\end{equation}

Combining \eqref{eq:Erd-3} with the definition of $k_d(g)$ and their relationship $\sum_{d = 1}^{\infty}k_d(g) = 0$, we have $k_2(g) > 0, k_1(g)< 0$ and moreover $1\leq k_2(g)\leq - k_1(g)\leq \pi_q(1)=3$.

Therefore, $\Phi(f) = \sigma(g)$ if and only if
\begin{equation}\label{eq:Erd-3-1}
\begin{split}
& \left\{\begin{array}{ll}
k_d(g) = m_d(f),  \textrm{for any $d \ge 3$,}\\
m_1(f)+3m_2(f) = k_1(g) + 3k_2(g),\\
1\leq k_2(g)\leq - k_1(g)\leq \pi_q(1)=3,\\
\sum_{d = 1}^{\infty}k_d(g) = 0.
\end{array}
\right.
\end{split}
\end{equation}

By \eqref{eq:Erd-3-1}, we can discuss case by case on the value of $k_1(g)$.

If $k_1(g) = -1$, we have $k_2(g) = 1$, which gives $\Phi(f) = \sigma(g) = \frac{3^2-1}{3-1} = (3-1)(3-1)$.

If $k_1(g) = -2$, we have $1 \leq k_2(g) \leq 2$, then there exists $d\geq 2 $ such that $k_{d}(g) = 1 $, which gives $\Phi(f) = \sigma(g) = \frac{3^2-1}{3-1}\cdot \frac{3^d-1}{3-1} = (3-1) (3^d-1)$.

If $k_1(g) = -3$, we have $1 \leq k_2(g) \leq 3$, then there exists $d_1, d_2 \geq 2 $ such that $k_{d_1}(g) = k_{d_2}(g) = 1 $, which gives $\Phi(f) = \sigma(g) = \frac{3^2-1}{3-1}\cdot\frac{3^{d_1}-1}{3-1}\cdot \frac{3^{d_2}-1}{3-1} = (3^{d_1}-1)(3^{d_2}-1)$.

Thus, we obtain $\Phi(\A) \cap \sigma(\A) = \{(3^{d_1}-1)(3^{d_2}-1)| d_1, d_2 \geq 1 \}$.
\end{proof}

\begin{theorem}\label{thm:Erd2}
Assume that $ q = 2 $. Then, we have
\begin{align*}
    \Phi(\A) \cap \sigma(\A) = & \{2^d-1| d \geq 2 \}\cup \{(2^2-1)(2^d-1)| d \geq 3, \ 2|d \textit{ or $3|d$} \}\\
    &\cup \{(2^2-1)(2^3-1)(2^d-1)| d \geq 3 \}\\
    &  \cup\{(2^{d_1}-1)(2^{d_2}-1)| d_1 \geq 2, d_2 \geq 3 \} \\
     &  \cup\{(2^2-1)(2^3-1)(2^{d_1}-1)(2^{d_2}-1)| d_1 \geq 3, d_2 \geq 4 \} \\
    &  \cup \{(2^2-1)(2^{d_1}-1)(2^{d_2}-1)| d_1, d_2 \geq 4, 2|d_1 \textit{ or $3|d_1$} \}\\
    & \cup \{(2^2-1)(2^{d_1}-1)(2^{d_2}-1)(2^{d_3}-1)| d_1 , d_2, d_3 \geq 4,\\
     & \ \ \ 2|d_1 \textit{ or $3|d_1$}  \}.
\end{align*}

\end{theorem}

\begin{proof}
For any $f, g \in \A$, we know that
$\Phi(f) = \sigma(g)$ if and only if \eqref{eq:Erd} holds. As in the proof of Theorem~ \ref{thm:same-E} (3), this is equivalent to that
$k_d(g) = m_d(f)$ for any $d \ge 7, k_4(g) = m_4(f), k_5(g) = m_5(f)$ and
\begin{align*}
   & m_2(f) + 2 m_6(f) = k_2(g) + 2k_6(g),\\
 & m_3(f) +  m_6(f) = k_3(g) + k_6(g).
\end{align*}

Considering the definition of $k_d(g), m_d(f)$ and $\sum_{d = 1}^{\infty}k_d(g) = 0$, we have $ 0 \leq m_2(f) \leq \pi_2(2)=1, 0 \leq m_3(f) \leq \pi_2(3) = 2, -2 \leq k_1(g) \leq 0, -1 \leq k_2(g)\leq 2, -2 \leq k_3(g)\leq 2, k_2(g) + k_3(g)\leq \pi_2(1)=2 $ and moreover $k_6(g) \geq 0, -5 \leq k_1(g) + k_2(g) + k_3(g) \leq -1$.
Note that $2^6 -1 = (2^2 - 1)^2(2^3 -1)$, then $k_2(g) = k_3(g) = -1 $ or $ -1 \leq k_2(g) \leq 1, 0 \leq k_3(g) \leq 2 $. Hence, $-4 \leq k_1(g) + k_2(g) + k_3(g) \leq -1$.

 Therefore, $\Phi(f) = \sigma(g)$ if and only if
\begin{equation}\label{eq:Erd-2}
\begin{split}
& \left\{\begin{array}{ll}
k_d(g) = m_d(f),  \textrm{for any $d \ge 7$,}, \\
k_4(g) = m_4(f), k_5(g) = m_5(f),\\
m_2(f) + 2 m_6(f) = k_2(g) + 2k_6(g),\\
m_3(f) +  m_6(f) = k_3(g) + k_6(g),\\
0\leq m_2(f)\leq 1, 0 \leq m_3(f) \leq  2,\\
-2 \leq k_1(g) \leq 0,  k_6(g) \geq 0,\\
 k_2(g) = k_3(g) = -1 \ \textit{or}\   -1 \leq k_2(g) \leq 1, 0 \leq k_3(g) \leq 2   \\
 -4 \leq k_1(g) + k_2(g) + k_3(g) \leq -1,\\
\sum_{d = 1}^{\infty}k_d(g) = 0.
\end{array}
\right.
\end{split}
\end{equation}

By \eqref{eq:Erd-2}, we discuss case by case on the value of $k_1(g) + k_2(g) + k_3(g)$.

If $k_1(g) = -2, k_2(g) = -1, k_3(g) = -1$, then $k_6(g) \geq 1$ and there exists $d_1, d_2, d_3 \geq 4 $ such that  $k_{d_1}(g) = k_{d_2}(g)  = k_{d_3}(g)= 1 $, which gives $\Phi(f) = \sigma(g) = \frac{2^{d_1}-1}{2-1}\cdot \frac{2^{d_2}-1}{2-1}\cdot \frac{2^{6}-1}{2^2-1}\cdot \frac{2^{d_2}-1}{2^3-1} = 3(2^{d_1}-1)(2^{d_2}-1)(2^{d_3}-1)$.

If $k_1(g) = -2, k_2(g) = -1, k_3(g) \geq 0$, then $k_6(g) \geq 1$ and there exists $d_1\geq 3, d_2\geq 4 $ such that  $k_{d_1}(g) = k_{d_2}(g) = 1 $, which gives $\Phi(f) = \sigma(g) = \frac{2^{d_1}-1}{2-1}\cdot \frac{2^{d_2}-1}{2-1}\cdot \frac{2^{6}-1}{2^2-1} = 3(2^3-1)(2^{d_1}-1)(2^{d_2}-1)$.

If $k_1(g) = -1, k_2(g) = -1, k_3(g) = -1 $, then $k_6(g) \geq 1$ and there exists  $d_1, d_2\geq 4 $ such that  $k_{d_1}(g) = k_{d_2}(g) = 1 $, which gives $\Phi(f) = \sigma(g) = \frac{2^{d_1}-1}{2-1}\cdot \frac{2^{d_2}-1}{2^2-1}\cdot \frac{2^{6}-1}{2^3-1} = 3(2^{d_1}-1)(2^{d_2}-1)$.

If $k_1(g) = -2, k_2(g) \geq 0, k_3(g) \geq 0$, then there exists $d_1\geq 2, d_2\geq 3 $ such that  $k_{d_1}(g) = k_{d_2}(g) = 1 $, which gives $\Phi(f) = \sigma(g) = \frac{2^{d_1}-1}{2-1}\cdot \frac{2^{d_2}-1}{2-1} = (2^{d_1}-1)(2^{d_2}-1)$.

If $k_1(g) = -1, k_2(g) = -1, k_3(g) \geq 0$, then $k_6(g) \geq 1$ and there exists $d\geq 3 $ such that  $k_{d}(g) = 1 $, which gives $\Phi(f) = \sigma(g) = \frac{2^{d}-1}{2-1}\cdot \frac{2^{6}-1}{2^2-1} = 3(2^3-1)(2^{d}-1)$.

If $k_1(g) = 0, k_2(g) = -1, k_3(g) = -1$, then $k_6(g) \geq 1$ and there exists $d\geq 4 $ such that  $k_{d}(g) = 1 $, which gives $\Phi(f) = \sigma(g) = \frac{2^{d}-1}{2^2-1}\cdot \frac{2^{6}-1}{2^3-1} = 3(2^{d}-1)$.

If $k_1(g) = -1, k_2(g) \geq 0, k_3(g) \geq 0$, then there exists $d\geq 2 $ such that  $k_{d}(g) = 1 $, which gives $\Phi(f) = \sigma(g) = \frac{2^{d}-1}{2-1} = 2^d-1$.

If $k_1(g) = 0, k_2(g) = -1, k_3(g) \geq 0$, then $k_6(g) = 1$, which gives $\Phi(f) = \sigma(g) = \frac{2^{6}-1}{2^2-1} = (2^2-1)(2^3-1)$.

 Thus, combining the above discussions with \eqref{eq:sig3},  we obtain \begin{align*}
    \Phi(\A) \cap \sigma(\A) = & \{2^d-1| d \geq 2 \}\cup \{(2^2-1)(2^d-1)| d \geq 3, \ 2|d \textit{ or $3|d$} \}\\
    &\cup \{(2^2-1)(2^3-1)(2^d-1)| d \geq 3 \}\\
    &  \cup\{(2^{d_1}-1)(2^{d_2}-1)| d_1 \geq 2, d_2 \geq 3 \} \\
     &  \cup\{(2^2-1)(2^3-1)(2^{d_1}-1)(2^{d_2}-1)| d_1 \geq 3, d_2 \geq 4 \} \\
    &  \cup \{(2^2-1)(2^{d_1}-1)(2^{d_2}-1)| d_1, d_2 \geq 4, 2|d_1 \textit{ or $3|d_1$} \}\\
    & \cup \{(2^2-1)(2^{d_1}-1)(2^{d_2}-1)(2^{d_3}-1)| d_1 , d_2, d_3 \geq 4,\\
     & \ \ \ 2|d_1 \textit{ or $3|d_1$}  \}.
\end{align*}
\end{proof}

\section{Distribution of the values of $\Phi$}
\label{sec:Dis-E}

 In this section, as an analogue of the integer case \cite{Ford1998},  we want to study the distribution of $\Phi(\A)$, that is, study the function $V(y)$, the number
of the values in $\Phi(\A)$ and not greater than $y$.

Define $\mathscr{V}(y) = \Phi(\A)\cap [1,y]$, then $V(y) = |\mathscr{V}(y)|$. Now we estimate the value of $V(y)$.

\begin{theorem}\label{thm:dis}
$V(y) \leq 2 q k \big(\frac{e^2}{2}\big)^{\frac{k}{2}}$, where $k = \lfloor\log_qy\rfloor$.
\end{theorem}
\begin{proof} For any $n\in\mathscr{V}(y)$, there exists $j, i $ and $m_1,m_2,\ldots,m_i$ such that
$$ n = q^j\prod_{d=1}^{i}\big(q^d-1\big)^{m_d} \leq y.$$

Since $q^j \leq y$, we have $j\leq k$. So, there are at most $k$ choices of $j$. Now we consider the remainder part $\prod_{d=1}^{i}\big(q^d-1\big)^{m_d}$.
Since $q^{d-1}\leq q^d-1$, we have
$$\prod_{d=1}^{i}q^{(d-1)m_d} \leq \prod_{d=1}^{i}\big(q^d-1\big)^{m_d} \leq y,$$
so, $i\leq k+1$ and $\sum_{d=1}^{k+1}(d-1)m_d \leq k$, which means
$(m_2,m_3,\ldots,m_{k+1})$ is a non-negative integer solution of the inequality $x_1 + 2x_2 + \ldots + kx_k \leq k$. By Lemma \ref{lem:S}, there are at most $2\big(\frac{e^2}{2}\big)^{\frac{k}{2}}$ choices of $(m_2,m_3,\ldots,m_{k+1})$ such that $\prod_{d=1}^{n}q^{(d-1)m_d}\leq y$.

Finally, combining $m_1\leq \pi_q(1)=q$ with the above discussion, we obtain
$$V(y) \leq 2 q k \big(\frac{e^2}{2}\big)^{\frac{k}{2}}.$$

\end{proof}

The following result suggests that almost all positive integers are not in $\Phi(\A)$.

\begin{corollary}
The natural density of $\Phi( \A )$ is zero, that is, $\lim_{y\rightarrow\infty}\frac{V(y)}{y} = 0$.

\end{corollary}

\begin{proof}
By Theorem~\ref{thm:dis}, we have $V(y) \leq 2q(\log_q y)(\frac{e}{\sqrt{2}})^{\log_q y} $. Note that
\begin{align*}
   \lim_{y\rightarrow\infty}\frac{2q(\log_q y)(\frac{e}{\sqrt{2}})^{\log_q y}}{y}& = \lim_{y\rightarrow\infty}\frac{2q(\log_q y)(\frac{e}{\sqrt{2}})^{\log_q y}}{q^{\log_q y}}\\
   &=\lim_{y\rightarrow\infty}\frac{2q\log_q y}{(\frac{\sqrt{2}q}{e})^{\log_q y}} = 0,
\end{align*}
where we use the fact that $(\frac{\sqrt{2}q}{e})^{\log_q y} > 1$ for any $q \geq 2$.
So, we obtain $\lim_{y\rightarrow\infty}\frac{V(y)}{y} = 0$.
\end{proof}

\section*{Acknowledgement}
For the research,  Xiumei Li was supported by the National Science Foundation of China Grant No.12001312; 
and Min Sha was supported  was supported by Key Laboratory of Applied Mathematics of Fujian Province University (Putian University) (NO. SX202201) and by the Guangdong Basic and Applied Basic Research Foundation (No. 2025A1515010635).

\end{document}